\documentclass[12pt,reqno,a4paper]{amsart}
\usepackage{fullpage,  amssymb, amsthm, amsmath, amsfonts, times}
\usepackage{enumitem}
\usepackage{mathrsfs}
\usepackage{mathtools}
\usepackage{centernot}
\usepackage{xfrac}
\usepackage{eufrak}
\usepackage{stmaryrd}
\usepackage{bm}
\usepackage[all,cmtip]{xy}
\usepackage[usenames,dvipsnames]{xcolor}
\usepackage[bottom = 1.5in, top = 1.0in, left=1.0in, right=1.0in]{geometry}
\usepackage{float}
\usepackage{hyperref}
\hypersetup{colorlinks=true, citecolor=blue, linkcolor=blue, urlcolor=blue, pdfstartview=FitH, pdfauthor=Baily-Loepp, pdftitle=Countable Precompletions}
\usepackage[backend=biber,style=alphabetic-verb,maxnames=50,giveninits=true]{biblatex}
\usepackage{cleveref}

\usepackage{tikz-cd}
\usetikzlibrary{positioning,arrows}
\usetikzlibrary{calc}

\usepackage{todonotes}

\begin{document}

\parskip0pt
\parindent10pt

\newenvironment{answer}{\color{Blue}}{\color{Black}}
\newenvironment{exercise}
{\color{Blue}\begin{exr}}{\end{exr}\color{Black}}

\theoremstyle{plain} %italicizes text
\numberwithin{equation}{section}

\newtheorem{thm}{Theorem}[section]
\newtheorem{theorem}[thm]{Theorem}
\newtheorem*{theorem*}{Theorem}
\newtheorem{proposition}[thm]{Proposition}
\newtheorem*{proposition*}{Proposition}
\newtheorem{porism}[thm]{Porism}
\newtheorem{lemma}[thm]{Lemma}
\crefname{cor}{cor}{corollaries}
\newtheorem{algorithm}[thm]{Algorithm}
\crefname{algorithm}{algorithm}{algorithms}
\newtheorem{conjecture}{Conjecture}
\newtheorem{que}[thm]{Question}
\newtheorem{corollary}[thm]{Corollary}
\newtheorem{conj}[thm]{Conjecture}
\newtheorem*{claim}{Claim}
\newtheorem*{cor*}{Corollary}
\newtheorem*{notate}{Notation}

\theoremstyle{definition}
\newtheorem*{conv}{Convention}
\newtheorem{definition}[thm]{Definition}
\crefname{definition}{definition}{definitions}
\newtheorem{example}{Example}
\newtheorem{remark}[thm]{Remark}
\crefname{remark}{remark}{remarks}

\newcommand{\doublerule}[1][.4pt]{%
  \noindent
  \makebox[0pt][l]{\rule[.7ex]{\linewidth}{#1}}%
  \rule[.3ex]{\linewidth}{#1}}

% \algnewcommand\algorithmicswitch{\textbf{switch}}
% \algnewcommand\algorithmiccase{\textbf{case}}
% \algnewcommand\algorithmicassert{\texttt{assert}}
% \algnewcommand\Assert[1]{\State \algorithmicassert(#1)}%
% % New "environments"
% \algdef{SE}[SWITCH]{Switch}{EndSwitch}[1]{\algorithmicswitch\ #1\ \algorithmicdo}{\algorithmicend\ \algorithmicswitch}%
% \algdef{SE}[CASE]{Case}{EndCase}[1]{\algorithmiccase\ #1}{\algorithmicend\ \algorithmiccase}%
% \algtext*{EndSwitch}%
% \algtext*{EndCase}%

\renewcommand{\mod}[1]{{\ifmmode\text{\rm\ (mod~$#1$)}\else\discretionary{}{}{\hbox{ }}\rm(mod~$#1$)\fi}}

\newcommand{\Spec}{\text{Spec}}
\newcommand{\Sing}{\text{Sing}}
\newcommand{\Min}{\text{Min}}
\newcommand{\Ass}{\text{Ass}}
\newcommand{\Assmax}{\text{Ass}_{\text{max}}}
\newcommand{\Ann}{\text{Ann}}
\newcommand{\nb}{\mathbf n}
\newcommand{\ns}{\mathrel{\unlhd}}
\newcommand{\tr}{\text{tr}}
\newcommand{\wt}[1]{\widetilde{#1}}
\newcommand{\wh}[1]{\widehat{#1}}
\newcommand{\cbrt}[1]{\sqrt[3]{#1}}
\newcommand{\floor}[1]{\left\lfloor#1\right\rfloor}
\newcommand{\abs}[1]{\left|#1\right|}
\newcommand{\ds}{\displaystyle}
\newcommand{\nn}{\nonumber}
\newcommand{\im}{\text{Im}}
\newcommand{\re}{\text{Re}}
\renewcommand{\ker}{\textup{ker }}
\renewcommand{\char}{\textup{char }}
\renewcommand{\Im}{\textup{Im }}
\renewcommand{\Re}{\textup{Re }}
\newcommand{\area}{\textup{area }}
\newcommand{\isom}
    {\ds \mathop{\longrightarrow}^{\sim}}
\renewcommand{\ni}{\noindent}
\renewcommand{\bar}{\overline}
\newcommand{\morph}[1]
    {\ds \mathop{\longrightarrow}^{#1}}

\newcommand{\Gal}{\textup{Gal}}
\newcommand{\Aut}{\textup{Aut}}
\newcommand{\Crypt}{\textup{Crypt}}
\newcommand{\disc}{\textup{disc}}
\newcommand{\sgn}{\textup{sgn}}
\newcommand{\del}{\partial}

\newcommand{\mattwo}[4]{
\begin{pmatrix} #1 & #2 \\ #3 & #4 \end{pmatrix}
}

\newcommand{\vtwo}[2]{
\begin{pmatrix} #1 \\ #2 \end{pmatrix}
}
\newcommand{\vthree}[3]{
\begin{pmatrix} #1 \\ #2 \\ #3 \end{pmatrix}
}
\newcommand{\vcol}[3]{
\begin{pmatrix} #1 \\ #2 \\ \vdots \\ #3 \end{pmatrix}
}

\newsymbol\dnd 232D

\newcommand*\wb[3]{
  {\fontsize{#1}{#2}\usefont{U}{webo}{xl}{n}#3}}

\newcommand{\one}{{\rm 1\hspace*{-0.4ex} \rule{0.1ex}{1.52ex}\hspace*{0.2ex}}}

\renewcommand{\a}{\alpha}
\renewcommand{\b}{\beta}
\newcommand{\g}{\gamma}
\renewcommand{\d}{\delta}
\renewcommand{\l}{\lambda}
\renewcommand{\o}{\omega}
\renewcommand{\t}{\theta}
\renewcommand{\k}{\kappa}
\newcommand{\ve}{\varepsilon}
\renewcommand{\emptyset}{\varnothing}

\newcommand{\C}{\mathbb C}
\newcommand{\D}{\mathbb D}
\newcommand{\E}{\mathbb E}
\newcommand{\F}{\mathbb F}
\newcommand{\N}{\mathbb N}
\newcommand{\Q}{\mathbb Q}
\newcommand{\R}{\mathbb R}
\newcommand{\Z}{\mathbb Z}

\newcommand{\Ac}{\mathcal A}
\newcommand{\Bc}{\mathcal B}
\newcommand{\Cc}{\mathcal C}
\newcommand{\Dc}{\mathcal D}
\newcommand{\Ec}{\mathcal E}
\newcommand{\Fc}{\mathcal F}
\newcommand{\Gc}{\mathcal G}
\newcommand{\Hc}{\mathcal H}
\newcommand{\Kc}{\mathcal K}
\newcommand{\Lc}{\mathcal L}
\newcommand{\Nc}{\mathcal N}
\newcommand{\Oc}{\mathcal O}
\newcommand{\Pc}{\mathcal P}
\newcommand{\Qc}{\mathcal Q}
\newcommand{\Rc}{\mathcal R}
\newcommand{\Sc}{\mathcal S}
\newcommand{\Tc}{\mathcal T}
\newcommand{\Uc}{\mathcal U}
\newcommand{\Xc}{\mathcal X}
\newcommand{\Zc}{\mathcal Z}

\newcommand{\af}{\mathfrak a}
\newcommand{\cf}{\mathfrak c}
\newcommand{\df}{\mathfrak d}
\newcommand{\ef}{\mathfrak e}
\newcommand{\ff}{\mathfrak f}
\newcommand{\gf}{\mathfrak g}
\newcommand{\hf}{\mathfrak h}
\newcommand{\kf}{\mathfrak k}
\newcommand{\lf}{\mathfrak l}
\newcommand{\mf}{\mathfrak m}
\newcommand{\nf}{\mathfrak n}
\newcommand{\of}{\mathfrak o}
\newcommand{\pf}{\mathfrak p}
\newcommand{\qf}{\mathfrak q}
\newcommand{\rf}{\mathfrak r}
\renewcommand{\sf}{\mathfrak s}
\newcommand{\tf}{\mathfrak t}
\newcommand{\uf}{\mathfrak u}
\newcommand{\zf}{\mathfrak z}
\renewcommand{\hat}{\widehat}

\newcommand{\lr}[1]{\langle #1 \rangle}
\newcommand{\0}{{\vec 0}}

\newcommand{\ignore}[1]{}

\newcommand*\circled[1]{\tikz[baseline=(char.base)]{
            \node[shape=circle,draw,inner sep=2pt] (char) {#1};}}

\newcommand*\squared[1]{\tikz[baseline=(char.base)]{
            \node[shape=rectangle,draw,inner sep=2pt] (char) {#1};}}

\title{Completions of Countable Excellent Local Rings in Equal Characteristic Zero}

\subjclass{13B35, 13F40, 13J10}

\author{B. Baily}
\address{Department of Mathematics\\
University of Michigan \\
Ann Arbor, MI, USA}
\email{bbaily@umich.edu}

\author{S. Loepp}
\address{Department of Mathematics and Statistics \\
Williams College \\
Williamstown, MA, USA}
\email{sloepp@williams.edu}
\maketitle
\begin{abstract}
We characterize which complete local (Noetherian) rings $T$ containing the rationals are the completion of a countable excellent local ring $S$. We also discuss the possibilities for the map from the minimal prime ideals of $T$ to the minimal prime ideals of $S$ and we prove some characterization-style results. 
\end{abstract}
\section{Introduction}

In this paper, we explore the question ``Given a complete local (Noetherian) ring, under what circumstances is it the completion of a countable excellent local ring?'' When the complete local ring is presumed to have equal characteristic zero, our list of conditions is both necessary and sufficient. Our result shows that most complete local rings containing the rationals are the completion of a countable excellent local ring. Since complete local rings are excellent, all complete local rings are the completion of an excellent local ring (namely themself), and so the condition ``countable" in our question is what makes it interesting.  If a complete local ring $T$ contains the rationals then, by Cohen's structure theorem, it is isomorphic to a ring of the form $K[[x_1, \ldots ,x_n]]/I$ where $K$ is a field and $I$ is an ideal of $K[[x_1, \ldots ,x_n]]$.  If $K$ is uncountable, then, by previous results, $T$ cannot be the completion of a countable local ring. If, on the other hand, $K$ is countable, and $I = (f_1, \ldots ,f_m)$ where each $f_i$ is an element of the polynomial ring $K[x_1, \ldots ,x_n]$, then $T$ is the completion of the countable excellent local ring $K[x_1, \ldots ,x_n]_{(x_1, \ldots ,x_n)}/(f_1, \ldots ,f_m)$. Hence, the most interesting application of our result is when $K$ is countable and $I$ cannot be generated by elements of $K[x_1, \ldots ,x_n]$.

Our main theorem follows a line of previous results characterizing the completions of local domains \cite{Lech}, local unique factorization domains \cite{Heitmann1994CompletionsOL}, countable local domains \cite{controlling}, excellent local reduced rings \cite{Arnosti}, countable excellent local domains \cite{countable-domain}, and uncountable local domains with countable prime spectra \cite{countable-spec}.
Among the most recent of these is a result in \cite{countable-domain} that characterizes when a complete local ring $T$ is the completion of a countable excellent local domain. 
\begin{theorem}[\cite{countable-domain}, Theorem 3.9]\label{thm:ctbl-domain}
Let $T$ be a complete local ring containing the rationals and let $\mf$ be the maximal ideal of $T$. Then $T$ is the completion of a countable excellent local domain $A$ if and only if all of the following apply.
\begin{itemize}
    \item $T$ is equidimensional;
    \item $T$ is reduced;
    \item $T/\mf$ is countable.
\end{itemize}
\end{theorem}
Our main result (\Cref{Thm:ctbl-excellent-gen}) is a generalization of this theorem. To see how this result can be generalized, we first rephrase the theorem statement. First, recall that a ring $A$ is a domain if and only if it is reduced and has a unique minimal prime ideal, allowing us to replace ``$A$ is a domain'' with ``$A$ is reduced and has a unique minimal prime.'' Second, recall that an excellent ring $A$ is reduced if and only if its completion is reduced. This allows us to move the requirement that $A$ is reduced into the hypothesis of the theorem. These observations allow us to reformulate this result into the framework we will use for the remainder of the paper.
\begin{theorem}[Reformulation of \Cref{thm:ctbl-domain}]
Let $T$ be a complete local \textit{reduced} ring containing the rationals and let $\mf$ be the maximal ideal of $T$. Then $T$ is the completion of a countable excellent local ring $A$ with one minimal prime ideal if and only if $T$ is equidimensional and $T/\mf$ is countable.
\end{theorem}
Our main result generalizes this theorem by considering the case where $A$ is required to have $k$ minimal prime ideals, where $k$ is a positive integer, and removes the assumption that $T$ is reduced. Before we state our main result, we introduce some notation and a definition.

When we write that $(T, \mf)$ is a local ring, we mean that $T$ is a local (Noetherian) ring with maximal ideal $\mf$. If $T$ is a ring, then $\Min(T)$ denotes the minimal prime ideals of $T$. We use $\widehat{A}$ to denote the completion of the local ring $A$ with respect to its maximal ideal.

\begin{definition}\label{defn:equidim-params}
Let $(T, \mf)$ be a complete local ring. We define an equivalence relation $\sim_1$ on $\Min(T)$ as follows. Let $\pf, \qf\in \Min(T)$. We say that $\pf\sim_1 \qf$ if and only if all of the following hold.
\begin{itemize}
    \item $T_\pf$ is a field;
    \item $T_\qf$ is a field;
    \item We have $\dim(T/\pf) = \dim(T/\qf)$.
\end{itemize}
Similarly, we say $\pf\sim_2 \qf$ if and only if $\dim(T/\pf) = \dim(T/\qf)$. We then define $d_1(T) = \Min(T)/\sim_1$ and $d_2(T) = \Min(T)/\sim_2$. 
\end{definition}
Note that when $T$ satisfies Serre's $R_0$ condition (in particular, when $T$ is a reduced ring), we have that $T_\pf$ is a field for all $\pf\in\Min(T)$ and thus $d_1(T) = d_2(T)$.

%\medskip

The main result of this paper is the following result.
\begin{theorem}\label{Thm:ctbl-excellent-gen}
Let $(T, \mf)$ be a complete local ring containing the rationals and let $k$ be a positive integer. 
\begin{itemize}
    \item If both of the following conditions hold, then there exists a countable excellent local ring $A$ such that $\hat{A} = T$ and $|\Min(A)| = k$.
    \begin{enumerate}
        \item $T/\mf$ is countable;
        \item $|d_1(T)|\leq k\leq |\Min(T)|.$
    \end{enumerate}
    \item If there exists a countable excellent local ring $A$ such that $\hat{A} = T$ and $|\Min(A)| = k$, then
    \begin{enumerate}
        \item $T/\mf$ is countable and
        \item $|d_2(T)|\leq k\leq |\Min(T)|$.
    \end{enumerate}
\end{itemize}
\end{theorem}
Though we are not able to give necessary and sufficient conditions \textit{for each $k$}, our result has two important corollaries that characterize completions of classes of rings.
\begin{corollary}\label{cor:ctbl-excellent-red}
Let $(T,\mf)$ be a complete local ring containing the rationals and satisfying Serre's $R_0$ condition. Let $k$ be a positive integer. Then $T$ is the completion of a countable excellent local ring with $k$ minimal prime ideals if and only if $T/\mf$ is countable and $|d_2(T)| \leq k\leq |\Min(T)|$.
\end{corollary}
\begin{corollary}\label{cor:ctbl-excellent-char}
Let $(T,\mf)$ be a complete local ring containing the rationals. The following are equivalent.
\begin{enumerate}
    \item $T$ is the completion of a countable excellent local ring with $|\Min(T)|$ minimal prime ideals;
    \item $T$ is the completion of a countable excellent local ring;
    \item $T/\mf$ is countable.
\end{enumerate}
\end{corollary}
% \begin{theorem}[General Case]
% Let $(T, \mf)$ be a complete, local ring containing the rationals. Let $d(T) = \{\dim(T/\pf T): \pf\in \Min(T)\}$, the set of all distinct coheights of minimal prime ideals of $T$. Let $D(T) = \{\dim(T/\pf T): \pf\in \Min(T), T_\pf \text{ is a field}\}\cup \{\qf: T_\qf \text{ is not a field}\}$. Then the following two statements hold.
% \begin{enumerate}
%     \item If $T/\mf$ is countable and $|D(T)|\leq k\leq |\Min(T)|$, then there exists a countable excellent local ring $A$ such that $\hat{A} = T$.
%     \item If there exists a countable excellent local ring $A$ such that $\hat{A} = T$, then $T/\mf$ is countable and $|d(T)|\leq k\leq |\Min(T)|$.
% \end{enumerate}
% \end{theorem}
Of the two parts of the main theorem, the more challenging of the two to prove is the first; that is, the construction of the ring $A$. The construction of this ring proceeds in two parts. First, in Section \ref{section-3}, given a complete local ring $T$ satisfying the hypothesis of the theorem and given a positive integer $k$, we build a countable local ring $S$ such that $\hat{S} = T$ and $|\Min(S)| = k$. In Section \ref{section-4}, we show that this ring can be extended to a countable excellent local ring $A\supseteq S$ without modifying the structure of the set of minimal prime ideals and without changing the completion. 

All rings in this paper are commutative with unity.  When we say that a ring is local, we mean that it is both Noetherian and it has exactly one maximal ideal.  For a ring that has exactly one maximal ideal and is not necessarily Noetherian, we use the term quasi-local.
\section{Preliminaries}\label{section-2}
%\subsection{Zero-Divisors}
%In this subsection we discuss the set of zero-divisors of a ring.
%\begin{proposition}\label{prop:min-ass}
%For any ring $R$, we have $\Min(R)\subseteq \Ass(R)$.
%\end{proposition}

%NOTE: We need to define quasi-local. Probably do that in the introduction -- last paragraph of the intro can be notation and conventions.  This will get rid of some of the remarks in the introduction.

In this section, we present background information needed to prove our main theorem.  As our main result is about excellent rings, we begin by recalling relevant definitions related to excellent rings.

\begin{definition}
%[\cite{Excellent}, Definition 1.1]
\label{defn:formal-fiber}
Let $(A, \mf)$ be a local ring and let $\pf\in \Spec(A)$.  The \textbf{formal fiber ring of $A$ at $\pf$} is the ring $\hat{A}\otimes_A k(\pf)$ where $k(\pf):= A_{\pf}/\pf A_\pf$.
\end{definition}
\begin{definition}
%[\cite{Excellent}, Definition 1.2]
\label{defn:geom-regular}
Let $K$ be a field and $B$ a $K$-algebra. $B$ is called \textbf{geometrically regular} over $K$ if for every finite field extension $L$ of $K$, the ring $B\otimes_K L$ is regular. 
\end{definition}
\begin{definition}[Grothendieck]\label{defn:excellent-local}
Let $(A, \mf)$ be a local ring. $A$ is \textbf{quasi-excellent} if its formal fiber rings are geometrically regular; that is, for each $\pf\in \Spec(A)$ and for every finite field extension $L$ of $k(\pf)$, the ring $\hat{A}\otimes_A L$ is regular. If $A$ is also universally catenary, then $A$ is \textbf{excellent}. 
\end{definition}
\begin{proposition}[\cite{Excellent}, Corollary 1.6]\label{prop:closed-variety}
    For an excellent local ring $(A,\mf)$, the singular locus of $A$ (that is, the set of prime ideals $\pf\in\Spec(A)$ such that $A_\pf$ is not a regular local ring) is closed in $\Spec(A)$.
\end{proposition}
% \begin{example}
% The following are some examples of excellent rings.
% \begin{enumerate}
%     \item Any complete local ring, including any field;
%     \item Any ring of polynomials in finitely-many variables over an excellent ring;
%     \item A quotient of an excellent ring;
%     \item A localization of an excellent ring at a prime ideal.
% \end{enumerate}
% \end{example}
% \begin{theorem}
% Let $A$ be an excellent ring. Let $\pf\in \Spec(A)$ and $\qf\in \Spec(\hat{A})$ such that $\qf\cap A = \pf$. For each of the below properties (*), the ring $A_\pf$ is (*) if and only if the ring $\hat{A}_\qf$ is (*).
% \begin{enumerate}
%     \item Normal;
%     \item Reduced;
%     \item Cohen-Macaulay;
%     \item Gorenstein;
% \end{enumerate}
% \end{theorem}
%\subsection{Rings and Their Precompletions}
We now state two results that will be used to prove our main theorem. We use Proposition \ref{proposition:the-machine} to show that our final ring has the desired completion, and we use Lemma \ref{lemma:excellency-over-Q} to show that our final ring is excellent.
\begin{proposition}\label{proposition:the-machine}
Let $(T,\mf)$ be a complete local ring with quasi-local subring $(R, R \cap \mf)$. Then $R$ is Noetherian and $\widehat R = T$ if and only if the natural map $R \to T/\mf^2$ is onto and $IT\cap R = I$ for every finitely-generated ideal $I$ of $R$.
\end{proposition}
\begin{proof}
For the backwards direction, see \cite{Heitmann1994CompletionsOL}, Proposition 1. For the forwards direction, see \cite{countable-domain}, Corollary 2.5.
\end{proof}
%The following lemma is re-stated to discard the assumption that $T$ is equidimensional and as such does not conclude that $A$ is excellent, but rather just quasi-excellent.
\begin{lemma}\label{lemma:excellency-over-Q}
Let $(T,\mf)$ be a complete local ring containing the rationals. Given a local subring $(A, A\cap \mf)$ of $T$ such that $\widehat A = T$, the ring $A$ is quasi-excellent if and only if, for every $\pf\in \Spec(A)$ and every $\qf\in \Spec(T)$ lying over $\pf$, the ring $(T/\pf T)_\qf$ is a regular local ring.
\end{lemma}
\begin{proof}
This lemma is proven in \cite{countable-domain} (Lemma 2.8) with the additional assumption that $T$ is equidimensional. When this hypothesis is removed, the only conclusion that is lost is the fact that $A$ is universally catenary.
\end{proof}
We end this section with two results on the cardinality of complete local rings and their quotient rings. The first is well-known and a short proof is contained within the listed reference.
\begin{lemma}\label{lemma:cardinality-m-2}
Let $(T, \mf)$ be a local ring. If $T/\mf$ is finite, then $T/\mf^n$ is finite for all $n$. If $T/\mf$ is infinite, then $|T/\mf| = |T/\mf^n|$ for all $n$. 
\end{lemma}
\begin{proof}
See \cite{controlling}, Lemma 2.12.
\end{proof}
\begin{lemma}[\cite{control-formal-fiber}, Lemma 2.2]\label{lemma:cardinality-dim}
Let $(T, \mf)$ be a complete local ring with $\dim(T)\geq 1$. If $\pf$ is a non-maximal prime ideal of $T$, then $|T/\pf| = |T| \geq |\R|$.
\end{lemma}
%The final result of this section concerns how the spectra of rings $S,R$ compare with one another when $S\subseteq R\subseteq \hat{S}$.
%We end this section by recalling the Going-Down Property and stating a useful lemma.
%\begin{definition}[Going-Down Property]
%Let $S\subset R$ be an injective ring map and $\pi: \Spec(R)\to \Spec(S)$ be given by $\pi(\pf) = \pf\cap S$. The map $S\to R$ is said to satisfy the ``Going-Down Property'' if for all $\pf\subsetneq \pf'$ in $\Spec(S)$ such that there exists $\qf'\in \Spec(R), \qf'\cap S = \pf'$, there also exists $\qf\in \Spec(R)$ such that $\qf\cap S = \pf$.
%\end{definition}
% \begin{lemma}[\cite{matsumura_1987}, Theorem 9.5]\label{lemma:going-down} Let $A$ be a ring and $B$ a flat $A$-algebra; then the going-down theorem holds between $A$ and $B$.
% % (NOTE: We need some additional hypotheses for this lemma.  Look this up in the reference.) Let $S\subseteq R\subseteq \hat{S}$. Then the map $S\to R$ satisfies the going-down property.
% \end{lemma}
% % \end{proof}

% NOTE:  Can we take out the above lemma?

\section{Building the Base Ring}\label{section-3}
The main aim of this section is, given a particular complete local ring $T$, to construct a base ring for $T$. That is, we
construct a countable local subring $S$ of $T$ such that $\widehat{S} = T$ and $\Min(S)$ has the desired structure. Many of the arguments we use for our construction are adapted from \cite{Arnosti}, and so we use definitions and results from there.
To begin, we recall two key definitions from \cite{Arnosti}.
%\subsection{Intersection-Preserving Subrings}
\begin{definition}[\cite{Arnosti}, Definition 1.1]
Let $(T, \mf)$ be a complete local ring. Let $\Cc = \{\qf_1, \dots, \qf_n\}$ be a finite collection of incomparable non-maximal prime ideals of $T$. A partition $\Pc = \{\Cc, \{\Cc_i\}_{i=1}^n\}$ of $\Cc$ is \textbf{feasible} if the following hold for all $\rf\in \text{Ass}(T)$.
\begin{enumerate}
    \item There exists some $\qf_j$ such that $\rf \subseteq \qf_j$.
    \item There exists exactly one $l$ such that when $\rf\subseteq \qf_j, \qf_j\in \Cc_l$.
\end{enumerate}
Note that when $\Cc = \Min(T)$, {\em all} partitions of $\Cc$ are feasible. Although to prove the results in \cite{Arnosti}, it is important to consider cases where $\Cc$ is not $\Min(T)$, in this article, we will exclusively be in the case where $\Cc = \Min(T)$.  Therefore, in this paper, all partitions of $\Cc$ will be feasible.
\end{definition}

%QUESTIONS:  We don't have Section 5 in this paper, right (it's referred to above)?  Does that mean we should simplify this example?  And, if so, maybe the notation could be simplified as well?  Can we just note that any partition on Min(T) is a feasible partition?  In the definition below, doesn't the partition need to be feasible?  I think we should think about leaving out the example altogether and following the Arnosti paper more closely for the rest of this section.  We need a separate definition for CIP-subring since it's not in the Arnosti paper.  Also, we should make sure we actually use all of the results we have stated here from the Arnosti paper.  It looks like we are changing SIP to CIP for some of these results.  We need to justify this. For this section (to get to the final lemma), it looks like we only need Lemmas 3.4, 3.9, 3.10, 3.11, and 3.12.

% \begin{rem}
%  For the remainder of this paper, let $(T,\mf)$ be a complete local ring containing $\Q$ and with $\dim(T)\geq 1$. Also let $\Pc$ be a feasible partition.  
%  \end{rem}
 %
%
%
%
%
\begin{definition}[\cite{Arnosti}, Definition 2.6]\label{defn:IP}
 Let  $(R, R\cap \mf)$ be a quasi-local subring of the complete local ring $(T,\mf)$. Let $\Pc = \{\Cc, \{\Cc_i\}_{i=1}^n\}$ be a feasible partition on a finite collection of incomparable non-maximal prime ideals of $T$. Then $R$ is called an \textbf{Intersection-Preserving subring} (abbreviated IP-subring)  of $T$
 %with respect to the partition $\Pc$ 
 if the following conditions hold:
\begin{enumerate}
    \item $R$ is infinite;
    \item \label{two} For any $\pf \in \Cc$, $R \cap \pf = R \cap \qf$ for any $\qf \in \Min(T)$ satisfying $\qf \subseteq \pf$;
    \item \label{three} For $\pf, \pf' \in \Cc$, $\pf, \pf' \in \Cc_i$ if and only if $R \cap \pf = R \cap \pf'$;
    \item For all $\pf\in \Cc, r\in \pf\cap R$ we have $\Ann_T(r)\not\subset \pf$.
\end{enumerate}
The ring $R$ is called a \textbf{small intersection preserving subring} (abbreviated SIP-subring) of $T$ if, additionally, $|R| <|T|$.
\end{definition}

If $\Cc = \Min(T)$, then condition (\ref{two}) automatically holds, so it need not be checked. One important property of IP-subrings is that, by condition (\ref{three}), two elements of $\Cc$ lie over the same prime ideal of the IP-subring if and only if they are in the same element of the partition.

Because we are interested in constructing countable excellent rings, it will be useful to replace the condition $|R| < |T|$ in the definition of SIP-subrings with the condition that $R$ is countable.  This motivates the following definition.

\begin{definition}
Let $(T,\mf)$ be a complete local ring and let $\Pc = \{\Cc, \{\Cc_i\}_{i=1}^n\}$ be a feasible partition on a finite collection of incomparable non-maximal prime ideals of $T$.  If $R$ is a countable IP-subring of $T$, we call $R$ a \textbf{countable intersection preserving subring} (abbreviated CIP-subring) of $T$.
\end{definition}
\begin{remark}\label{rmk:precon}
For many results in this section, we will assume that $(T, \mf)$ is a complete local ring, $\dim T\geq 1$ and that $\Pc = \{\Cc, \{\Cc_1, \dots, \Cc_n\}\}$ is a feasible partition of a finite collection of incomparable non-maximal prime ideals of $T$.
\end{remark}
%
%
%
% \begin{definition}[\cite{Arnosti}, Definition 3.9]\label{defn:Semi-CIP}
%  A quasi-local ring $(R, R\cap \mf)\subset T$ is semi-CIP if the following hold.
% \begin{enumerate}
%     \item $R$ is countably infinite;
%     \item 
%     \item For all $\pf\in \Cc, r\in \pf\cap R$ we have $\Ann_T(r)\not\subset \pf$.
% \end{enumerate}
% \end{definition}
%
%
%
%
%
%\subsection{Adapting Techniques from \cite{Arnosti}
We now state several lemmas based on results from \cite{Arnosti} that will be key towards building our base ring.
\begin{lemma}\label{lemma:Arnosti-unioning}
Let $T, \Pc$ be as in \Cref{rmk:precon}. Let $(R_i)_{i\in \Z^+}$ be a countable chain of CIP-subrings of $T$. Then $R:=\bigcup_{i\in \Z^+}R_i$ is a CIP-subring of $T$.
\end{lemma}

\begin{proof}
By Lemma \ref{lemma:cardinality-dim}, $T$ is uncountable and so each $R_i$ is an SIP-subring of $T$. By Lemma 3.1 in \cite{Arnosti}, $R$ is an IP-subring of $T$.  As $R$ is a countable union of countable sets, it is countable, and we have that $R$ is a CIP-subring of $T$.
\end{proof}
%
%
%
%
%

%
%

%
%

% \begin{proof}
% Apply \Cref{lemma:Arnosti-transc} with $q = 1$. Then $q\notin \pf$ for every $\pf\in \Cc$, so we can shoose $t'\in J$ such that $t+t'+\pf\in T/\pf$ is transcendental over $R/(\pf\cap R)$ for every $\pf\in \Cc$. By \Cref{lemma:Arnosti-adjoining,lemma:Arnosti-localization}, the ring $R[t+t']_{R[t+t']\cap \mf}$ is a CIP-subring of $T$. Furthermore, we have $t+J = t+t'+J = \pi(t+t')$.
% \end{proof}
% \begin{lemma}[$\dag$ \cite{Arnosti}, Lemma 3.7]\label{lemma:Arnosti-image-2}
% Let $R$ be a CIP-subring of $T$. Then, for any finitely-generated ideal $I$ of $R$ and any $c\in IT\cap R$, there exists a subring $S$ of $T$ with the following properties:
% \begin{enumerate}
%     \item $R\subseteq S$;
%     \item $S$ is a CIP-subring of $T$;
%     \item $|S| = |R|$;
%     \item $c\in IS$.
% \end{enumerate}
% \end{lemma}
%
%
%
%
%
\begin{lemma}\label{lemma:upwards-closed}
Let $S_0\subseteq S_1\subseteq \dots$ be a countable chain of subrings of a ring $R$ and write $S = \bigcup_{i\geq 0} S_i$. If $JR\cap S_i = J$ for each $i$ and for every finitely generated ideal $J$ of $S_i$, then $IR\cap S = I$ for every finitely generated ideal $I$ of $S$.
\end{lemma}
\begin{proof}
Let $I = (a_1, \dots, a_n)$ be an ideal of $S$ and let $c\in IR\cap S$. Choose $j$ such that $a_1, \dots, a_n, c\in S_j$. Then we have
$c\in (a_1, \dots, a_n)R\cap S_j = (a_1, \dots, a_n)S_j \subseteq (a_1,\dots, a_n)S = I$, hence $c\in I$. We conclude $IR\cap S = I$. 
\end{proof}
\begin{lemma}\label{lemma:Arnosti-image-3}
Let $T, \Pc$ be as in \Cref{rmk:precon}. Let $J$ be an ideal of $T$ such that $J\not\subseteq \pf$ for all $\pf\in \Cc$ and let $u+J\in T/J$. If $R$ is a CIP-subring of $T$, then there exists a CIP-subring $S$ of $T$ such that:
\begin{enumerate}
    \item $R\subseteq S\subset T$;
    \item $u+J$ is in the image of $\pi: S\to T/J$;
    \item For every finitely generated ideal $I$ of $S$, we have $IT\cap S = I$.
\end{enumerate}
\end{lemma}

\begin{proof}
By Lemma \ref{lemma:cardinality-dim}, $T$ is uncountable and so $R$ is an SIP-subring of $T$.  By Lemma 3.8 in \cite{Arnosti}, there is an SIP-subring $S$ of $T$ satisfying all three conditions of the lemma as well as the condition that $|S| = |R|$. Hence, $S$ is the desired CIP-subring of $T$.
\end{proof}
% \begin{proof}
% First we use \Cref{lemma:Arnosti-image-1} to find $S_0\supseteq R$ such that $u+J$ is in the image of $\pi$. For $i\geq 0$, let $\Omega_i = \{(I, c): I\subseteq S_i\text{ finitely generated, }c\in IT\cap S_i\}$, index $\Omega$ as $\{(I_j, c_j)\}_{j\geq 0}$
% and set $R_i^0 = S_i$. For each $i,j\geq 0$, we apply \Cref{lemma:Arnosti-image-2} with $J = I_j, c = c_j, R = R_i^j$ and call the resulting ring $S_i^j$. We then define $S_{i+1} = \bigcup_{j\geq 0} R_i^j$. By construction, we have $IT\cap S_i\subset S_{i+1}$ for all $i\geq 0$, and by \Cref{lemma:Arnosti-unioning} each ring $S_i$ is a CIP-subring of $T$. Finally, let $S:= \bigcup_{i\geq 0} S_i$. Another application of \Cref{lemma:Arnosti-unioning} gives that $S$ is a CIP-subring; we will show that this ring $S$ satisfies the desired conditions.
% \begin{enumerate}
%     \item We have $R \subseteq S_0 \subseteq S$.
%     \item Since $u\in S_0\subseteq S$, we clearly have $u+J$ in the image of $\pi$.    
%     \item This follows from \Cref{lemma:upwards-closed}.
% \end{enumerate}
% \end{proof}
%
%
%
%
%
% \begin{lemma}[\cite{Arnosti}, Lemma 3.10]\label{lemma:Arnosti-growing-semi-CIP}
% Fix some $1\leq i\leq m$. Let $p_i\in \pf_{ij}$ for all $1\leq j\leq n_i$ and $p_i\notin \pf_{kl}$ for all $k\neq i$. Suppose further that $\Ann_T(p_i)\subsetneq \pf_{ij}$ for any $1\leq j\leq n_i$. Then there exists $u\in T^\times$ such that $R[up_i]_{R[up_i]\cap \mf}$ is a semi-CIP subring of $T$.
% \end{lemma}
%
%
%
%
%
Now, a brief lemma that is alluded to in \cite{Arnosti}, but not proven directly.
\begin{lemma}\label{Lemma:Precompletion-Cover}
Let $T, \Pc$ be as in \Cref{rmk:precon}. Let $R$ be a CIP-subring of $T$ and suppose $T/\mf$ is countable. Then there exists a Noetherian CIP-subring $S$ of $T$ such that $R\subseteq S\subset T$ and $\widehat S = T$. 
\end{lemma}
\begin{proof}
To begin, let $R_0 = R$. By Lemma \ref{lemma:cardinality-m-2}, $T/\mf^2$ is countable. Index the elements of $T/\mf^2$ as $(u_i + \mf^2)_{i\in \Z^+}$. For each $i \geq 0$, let $R_{i+1}$ be the CIP-subring of $T$ from \Cref{lemma:Arnosti-image-3} using $R = R_i$, $J = \mf^2$, and $u=u_i$. Since $\mf^2$ is only contained in one prime ideal, the ideal $\mf$, it satisfies the hypothesis of the lemma. %and $R_{i+1}$ is a CIP-subring of $T$. 
%with $|R_{i+1}| = |R_i|$. 
Also note that $R_i \subseteq R_{i + 1}$.  Now, define $S$ to be:
\[
S: = \bigcup_{i=0}^\infty R_i
\]
By Lemma \ref{lemma:Arnosti-unioning}, $S$ is a CIP-subring of $T$. 
%To see that $S$ is quasi-local, observe that the ideal $S\cap \mf$ is maximal in $S$ and every element $u\in S\setminus \mf$ lives in some $S_i\setminus \mf$, hence $u$ is a unit in $S_i$ and thus also a unit in $S$. Thus $(S, S\cap \mf)$ is quasi-local.

To see that $S$ is Noetherian and $\hat{S}=T$, we use Proposition \ref{proposition:the-machine}. The map $S\to T/\mf^2$ is surjective by construction and by \Cref{lemma:upwards-closed} we have that $IT\cap S = I$ for all finitely generated ideals $I$ of $S$. We therefore conclude that $\hat{S} = T$.
\end{proof}
\begin{lemma}\label{lemma:CIP-existence}
Let $T, \Pc$ be as in \Cref{rmk:precon}. Suppose further that $T$ contains $\Q$. Then there exists a CIP-subring $S$ of $T$.
\end{lemma}
\begin{proof}
By Lemma 3.11 in \cite{Arnosti}, there is an SIP-subring $S$ of $T$. In the proof of that lemma, the ring $S$ is constructed by starting with $\Q$, adjoining an element of $T$ and then localizing, and doing this process finitely many times.  It follows that $S$ is countable, and so it is a CIP-subring of $T$.
%Note that while our definition CIP-subring calls for $S$ to be countable, the statement of this lemma in \cite{Arnosti} only requires $|S|<|T|$. Fortunately, the ring constructed by \cite{Arnosti} is given by $S = \Q[t_1, \dots, t_n]_{\Q[t_1, \dots, t_n]\cap \mf}$ where the $t_i$ live in $T$, hence $S$ is countable.
\end{proof}
%
%
%
%
%
%We now proceed to the construction. Despite the similarities in the initial portion of our argument to that of \cite{Arnosti}, we note two key differences. First, we do not suppose that $\Q\subset T$, and second, we do not attempt to control the formal fibers of the $S$ to the degree done in \cite{Arnosti}; this concession allows us to ensure that the base ring is countable.

%\medskip

%NOTE:  I don't understand the comment above about how we do not suppose that Q is in T.
%
%
%

We are now ready to construct our base ring.
\begin{proposition}\label{prop:base-ring}
Let $T, \Pc$ be as in \Cref{rmk:precon}. Suppose further that $T$ contains $\Q$ and $T/\mf$ is countable. Then $T$ has a Noetherian CIP-subring $S$ such that $\hat{S} = T$.
\end{proposition}
\begin{proof}
We first construct a CIP-subring $R$ of $T$ using \Cref{lemma:CIP-existence}. We then use \Cref{Lemma:Precompletion-Cover} on $R$ to produce a Noetherian CIP-subring $S$ of $T$ such that $\hat{S} = T$. 
\end{proof}

We end this section with an interesting corollary of Proposition \ref{prop:base-ring}.

\begin{corollary}\label{notexcellent}
Let $(T,\mf)$ be a complete local ring containing the rationals.  Suppose further that $\dim T\geq 1$ and $T/\mf$ is countable.  Let $k$ be a positive integer such that $k \leq |\Min(T)|$.  Then there is a countable local ring $A$ such that $\widehat{A} = T$ and $|\Min(A)| = k$.
\end{corollary}

\begin{proof}
Let $\Pc = \{\Cc, \{\Cc_1, \dots, \Cc_k\}\}$ be any partition of $\Min(T)$ and note that $\Pc$ is a feasible partition.  Let $A$ be the Noetherian CIP-subring of $T$ given by Proposition \ref{prop:base-ring}.  Then $A$ is countable and $\widehat{A} = T$.  Since $T$ is a faithfully flat extension of $A$, if $\pf$ is a minimal prime ideal of $T$, then $\pf \cap A$ is a minimal prime ideal of $A$.  Moreover, given a minimal prime ideal $\qf$ of $A$, there is a minimal prime ideal $\pf$ of $T$ such that $\qf = \pf \cap A$.  It follows that, since $A$ is a IP-subring of $T$ and the partition $\Pc$ has $k$ elements, $|\Min(A)| = k$.
\end{proof}

In the proof of the previous corollary, it is shown that there is a surjective map from $\Min(T)$ to $\Min(A)$.  It is interesting to note that there is no restriction on the fibers of this map. That is, $\Pc$ can be {\em any} partition of the minimal prime ideals of $T$, and the fibers of the map from $\Min(T)$ to $\Min(A)$ will be exactly the elements of $\Pc$.  

\section{Countable Excellent Rings}\label{section-4}
The ring $A$ given in Corollary \ref{notexcellent} satisfies all properties we want $A$ to satisfy in the first part of Theorem \ref{Thm:ctbl-excellent-gen} with the exception that $A$ is not excellent.  In this section, we construct an {\em excellent} ring that satisfies all of our desired properties.  To do this, we start with the base ring given by Proposition \ref{prop:base-ring}, and we adjoin countably many elements from $T$.  In particular, we adjoin generating sets from carefully chosen prime ideals of $T$, which will allow us to conclude that our final ring is excellent. We begin by recalling the definition of nonsingular and singular prime ideals of a ring.

\begin{definition}
%[Singular and Nonsingular Prime Ideals]
\label{defn:nonsingularnew}
We say that a prime ideal $\pf$ of the ring $T$ is \textbf{nonsingular} if $T_\pf$ is a regular local ring. Otherwise we say that $\pf$ is \textbf{singular}.
\end{definition}

To construct our final ring, we treat nonsingular minimal prime ideals of $T$ slightly differently than singular ones.  In particular, we modify the definition of intersection-preserving subring from Section \ref{section-3} as follows.

\begin{definition}\label{defn:GIPNew}
Let $(T,\mf)$ be a complete local ring with $\dim T \geq 1$. Let $\Pc = \{\Cc_1, \dots, \Cc_n\}$ be a partition on $\Min(T)$ such that if $\pf\in \Cc_i$ for some $i \in \{1,2, \ldots ,n\}$ and $\pf$ is singular then $\Cc_i = \{\pf\}$. Let $(S, S\cap \mf)$ be a quasi-local subring of $T$.  If the following three conditions hold, we call $S$ a \textbf{generalized intersection-preserving subring} (abbreviated GIP-subring) of $T$.
\begin{enumerate}
    \item $S$ is infinite.
    \item \label{twonew} If $\pf,\pf' \in \Min(T)$, then $\pf \cap S = \pf' \cap S$ if and only if there is an $i \in \{1,2, \ldots ,n\}$ such that $\pf \in \Cc_i$ and $\pf' \in \Cc_i$.
    \item \label{threenew} If $\pf\in \Min(T)$ is nonsingular and $a\in\pf \cap S$, then $\Ann_T(a)\not\subseteq \pf$.
\end{enumerate}
If $S$ is also countable, then we call $S$ a \textbf{countable generalized intersection-preserving subring} (abbreviated CGIP-subring) of $T$.
\end{definition}

Note that if $\Pc = \{\Cc_1, \dots, \Cc_n\}$ is a partition on $\Min(T)$, then any IP-subring of $T$ is also a GIP-subring of $T$ and any CIP-subring of $T$ is also a CGIP-subring of $T$.

\begin{remark}
In this paper, the tools developed for GIP-subrings are used to grow a ring while maintaining the structure of the ring's minimal prime ideals. In particular, we begin with a pair of local rings $S,T$ such that $\hat{S} = T$. For a ring extension $S\subseteq R\subseteq T$ (this is the only setting in which we will be considering GIP-subrings), condition \ref{twonew} of Definition \ref{defn:GIPNew} is equivalent to the statement that the restriction of the map $\Spec(R)\to \Spec(S)$ to $\Min(R)\to \Min(S)$ is bijective. Though this is a more helpful framework for visualizing the structure of the ring inclusions and for discovering new results, it is less practical for the task of formalizing results.
\end{remark}

In constructing our final ring, Lemma \ref{localize} is very useful. It states that if $S$ is a countably infinite subring of a complete local ring $T$ and $u \in T$ is chosen so that $S[u]$ satisfies conditions \ref{twonew} and \ref{three} of Definition \ref{defn:GIPNew}, then the ring $R = S[u]_{(S[u] \cap \mf)}$ is a CGIP-subring of $T$.

\begin{lemma}\label{localize}
Let $(T,\mf)$ be a complete local ring with $\dim T \geq 1$. Let $\Pc = \{\Cc_1, \dots, \Cc_n\}$ be a partition on $\Min(T)$ such that if $\pf\in \Cc_i$ for some $i \in \{1,2, \ldots ,n\}$ and $\pf$ is singular then $\Cc_i = \{\pf\}$.  Suppose $S$ is a countably infinite subring of $T$ and $u \in T$ such that $S[u]$ satisfies conditions \ref{twonew} and \ref{threenew} of Definition \ref{defn:GIPNew}.  Then $R = S[u]_{(S[u] \cap \mf)}$ is a CGIP-subring of $T$.
\end{lemma}

\begin{proof}
Since $S$ is countable and infinite, so is $R$.  It remains to show that conditions \ref{twonew} and \ref{threenew} of Definition \ref{defn:GIPNew} hold for $R$.  Suppose $\pf,\pf' \in \Min(T)$ with $\pf \cap R = \pf' \cap R$. Then $\pf \cap S[u] = \pf' \cap S[u]$ and so there is an $i \in \{1,2, \ldots ,n\}$ such that $\pf \in \Cc_i$ and $\pf' \in \Cc_i$. Conversely, suppose there is an $i \in \{1,2, \ldots ,n\}$ such that $\pf \in \Cc_i$ and $\pf' \in \Cc_i$.  Then $\pf \cap S[u] = \pf' \cap S[u]$.  If $f/g \in \pf \cap R$, then $f \in \pf \cap S[u] = \pf' \cap S[u]$ and so $f/g \in \pf' \cap R$. It follows that $\pf \cap R \subseteq \pf' \cap R$.  Containment the other way follows similarly.  Finally, let $\pf \in \Min(T)$ be nonsingular, and let $f/g \in \pf \cap R$ where $f,g \in S[u]$.  Then $f \in \pf \cap S[u]$ and so $\Ann_T(f)\not\subseteq \pf$.  It follows that $\Ann_T(f/g)\not\subseteq \pf$.
\end{proof}

Our next step is to find sufficient conditions on $u \in T$ so that, if $S$ is a countably infinite subring of a complete local ring $T$, then $S[u]$ satisfies conditions \ref{twonew} and \ref{three} of Definition \ref{defn:GIPNew}.  The next two lemmas, taken from \cite{Arnosti}, help us do so.

\begin{lemma}[\cite{Arnosti}, Lemma 3.3 ]\label{lemma:Arnosti-adjoining-lemma}
%Let $(T, \mf)$ be a complete local ring of positive dimension, and let $\Pc = \{\Cc, \{\Cc_i\}_{i=1}^n\}$ be a feasible partition on a finite collection of incomparable non-maximal prime ideals of $T$. 
Let $R$ be a subring of the complete local ring $T$. Suppose that $\pf, \qf\in \Spec(T)$ such that $\pf\cap R = \qf\cap R$. Let $u\in T$ and suppose that $u+\pf\in T/\pf$ is transcendental over $R/(R \cap \pf)$ and likewise $u + \qf \in T/\qf$ is transcendental over $R/(R \cap \qf)$. Then $\pf\cap R[u] = \qf\cap R[u]$ and if $\Ann_T(a)\not\subseteq \pf$ for all $a\in R\cap \pf$, then $\Ann_T(f)\not\subseteq \pf$ for all $f\in R[u]\cap \pf$.
\end{lemma}

\begin{lemma}[\cite{Arnosti}, Lemma 3.5]\label{lemma:Arnosti-transcnew}
Let $(T,\mf)$ be a complete local ring of positive dimension and suppose $T$ contains the rationals. Let $R$ be a subring of $T$ such that $|R|<|T|$ and let $C$ be a finite set of incomparable non-maximal prime ideals of $T$. Let $J$ be an ideal of $T$ not contained in any prime of $C$. Let $t,q\in T$. Then there exists an element $t'\in J$ such that, for every $\pf\in C$ with $q\notin \pf$, we have $t+qt'+\pf\in T/\pf$ is transcendental over $R/(\pf\cap R)$. In addition, if $\qf\in \Min(T), \pf\in C, \qf\subset \pf, q\notin \pf$ and $R\cap \pf = R\cap \qf$, then $t+qt'+\qf\in T/\qf$ is transcendental over $R/(R\cap \qf)$.
\end{lemma}

Recall that, for the countable excellent local ring we construct, we want its completion to be our given complete local ring $T$.  To make this happen, we use Proposition \ref{proposition:the-machine}.  In particular, if $A \subseteq T$ is our final ring, we want to ensure that $IT \cap A = I$ for all finitely generated ideals $I$ of $A$.  Lemma \ref{lemma:CGIP-close-upnew} is a first step for doing this.

\begin{lemma}\label{lemma:CGIP-close-upnew}
Let $(T,\mf)$ be a complete local ring containing the rationals and suppose $\dim T \geq 1$. Let $\Pc = \{\Cc_1, \dots, \Cc_n\}$ be a partition on $\Min(T)$ such that if $\pf\in \Cc_i$ for some $i \in \{1,2, \ldots ,n\}$ and $\pf$ is singular then $\Cc_i = \{\pf\}$. Let $(S,S \cap \mf)$ be a CGIP-subring of $T$.  Suppose $I$ is a finitely generated ideal of $S$.  If $c \in IT \cap S$, there exists a CGIP-subring $R$ of $T$ such that $S \subseteq R$ and $c \in IR$.
\end{lemma}

\begin{proof}
The proof of this lemma follows closely the proof of Lemma 3.7 in \cite{Arnosti}. Let $I = (a_1,a_2, \ldots ,a_m)$.  We induct on $m$.  Let $m = 1$.  Then $I = (a)$, and we have $c = au$ for some $u \in T$.  Let $E$ denote the set of elements of $\Min(T)$ that are both nonsingular and that contain $a$ and suppose $E$ is not empty. If $\pf \in E$, then $a \in \pf \cap S$ and so $\Ann_T(a)\not\subseteq \pf$. By the prime avoidance theorem, $\Ann_T(a)\not\subseteq \bigcup_{\pf \in E} \pf$.  Let $q \in \Ann_T(a)$ with $q \not\in \bigcup_{\pf \in E} \pf$.  Since $T$ is a complete local ring with positive dimension, we have by Lemma \ref{lemma:cardinality-dim} that $T$ is uncountable, and so $|S| < |T|$.  Use Lemma \ref{lemma:Arnosti-transcnew} with $C = E$ to find $t \in T$ such that $u + qt + \pf \in T/\pf$ is transcendental over $S/(\pf \cap S)$ for every $\pf \in E$.  Let $u' = u + qt$.  If $E$ is empty, let $u' = u$.  We claim that $R = S[u']_{(S[u']\cap \mf)}$ is the desired CGIP-subring of $T$.  Note that $au' = a(u + qt) = au = c$ and so $c \in aR.$  By Lemma \ref{localize}, the result follows if we can show that conditions \ref{twonew} and \ref{threenew} of Definition \ref{defn:GIPNew} hold for $S[u']$.  

Suppose $\pf, \pf' \in \Min(T)$ such that $\pf \cap S[u'] = \pf' \cap S[u']$.  Then $\pf \cap S = \pf' \cap S$ and so there is an $i \in \{1,2, \ldots ,n\}$ such that $\pf \in \Cc_i$ and $\pf' \in \Cc_i$.  Conversely, suppose $\pf,\pf' \in \Min(T)$ and there is an $i \in \{1,2, \ldots ,n\}$ such that $\pf \in \Cc_i$ and $\pf' \in \Cc_i$. Then $\pf \cap S = \pf' \cap S$.  We show $\pf \cap S[u'] = \pf' \cap S[u']$ by considering three cases.

First, suppose $\pf$ is singular.  Then $\pf = \pf'$ and so $\pf \cap S[u'] = \pf' \cap S[u']$.  Second, suppose $\pf$ is nonsingular and $a \not\in \pf$.  Since $\pf \cap S = \pf' \cap S$, $a \not\in \pf'$.  Let $f \in \pf \cap S[u']$ with $f = s_r(u')^r + \cdots + s_1u' + s_0$ where $s_i \in S$.  Then $a^rf \in \pf \cap S = \pf' \cap S$.  Since $a \not\in \pf'$, we have $f \in \pf'$ and it follows that $\pf \cap S[u'] \subseteq \pf' \cap S[u']$.  Showing containment the other way is a similar argument and so $\pf \cap S[u'] = \pf' \cap S[u']$.  Finally, suppose $\pf$ is nonsingular and $a \in \pf$ (so, in particular, $E$ is not empty).  Then $a \in \pf \cap S = \pf' \cap S$ and we have $a \in \pf'$.  Hence, $\pf,\pf' \in E$ and so $u' + \pf \in T/\pf$ is transcendental over $S /(\pf \cap S)$ and $u' + \pf' \in T/\pf'$ is transcendental over $S /(\pf' \cap S)$. By Lemma \ref{lemma:Arnosti-adjoining-lemma}, $\pf \cap S[u'] = \pf' \cap S[u']$.
%Let $f \in \pf \cap S[u']$.  Then $f = s_r(u')^r + \cdots + s_1u' + s_0 \in \pf$.  It follows that $s_i \in \pf \cap S = \pf' \cap S$ for all $i$. Hence, $\pf \cap S[u'] \subseteq \pf' \cap S[u']$.  Showing containment the other way is a similar argument and so $\pf \cap S[u'] = \pf' \cap S[u']$.  

We have left to show that condition \ref{threenew} of Definition \ref{defn:GIPNew} holds for $S[u']$. Let $\pf \in \Min(T)$ be nonsingular and let $f \in \pf \cap S[u']$.  First suppose $a \not\in \pf$.  Now, there is a positive integer $r$ such that $a^rf \in \pf \cap S$.  Hence, there is a $w \in \Ann_T(a^rf)$ such that $w \not\in \pf$.  Therefore, $wa^rf = 0$ and we have $wa^r \in \Ann_T(f)$ with $wa^r \not\in \pf$.  Thus, $\Ann_T(f)\not\subseteq \pf$.  Now suppose $a \in \pf$.  Then $\pf \in E$, and so $u' + \pf \in T/\pf$ is transcendental over $S/(\pf \cap S)$. By Lemma \ref{lemma:Arnosti-adjoining-lemma}, we have that condition \ref{threenew} of Definition \ref{defn:GIPNew} holds for $S[u']$.
%Write $f = s_r(u')^r + \cdots + s_1u' + s_0 \in \pf$. By our choice of $u'$, $s_i \in \pf \cap S$ for all $i$.  Since $S$ is a CGIP-subring of $T$, for every $i$, there is a $v_i \in \Ann_T(s_i)$ with $v_i \not\in \pf$. Then $v = \prod_{i = 1}^r v_i \in \Ann_T(f)$ and $v \not\in \pf$.  It follows that $\Ann_T(f)\not\subseteq \pf$ and so condition \ref{threenew} holds for $S[u']$.  
This completes the base case of our induction.

\medskip

Now let $I = (a_1,a_2, \ldots, a_m)$ with $m \geq 2$ and suppose the result holds for all ideals generated by fewer then $m$ elements.  It is shown in the proof of Lemma 3.7 in \cite{Arnosti} that, since $T$ contains the rationals, we can find generators for $I$ such that given any $\pf \in \Min(T)$, $a_1 \in \pf$ if and only if $a_2 \in \pf$.  So we assume this condition holds for $a_1$ and $a_2$.  Now, $c = a_1t_1 + a_2t_2 + \cdots + a_mt_m$ for some $t_i \in T$.  Let $E$ be the set of minimal prime ideals of $T$ that are both nonsingular and contain $a_1$ and assume that $E$ is not empty.  Let $\pf \in E$.  Since $a_1 \in \pf \cap S$, $\Ann_T(a_1) \not\subseteq \pf$.  By the prime avoidance theorem, $\Ann_T(a_1) \not\subseteq \bigcup_{\pf \in E} \pf$.  Let $q \in \Ann_T(a_1)$ with $q \not\in \bigcup_{\pf \in E} \pf.$  Let $F$ be the set of minimal prime ideals of $T$ that are both nonsingular and do not contain $a_1$ and
% the removable sentence
let $\{\pf_1, \dots, \pf_k\}$ enumerate the elements of $F$ not containing $q$. For each $\pf_i$, use prime avoidance to find an element $p_i\in \pf_i$ not contained in any other minimal prime. Since we may replace $q$ with $qp_1\dots p_k$, we may assume without loss of generality that $q\in \pf$ for all $\pf\in F$.
% end removable
Use Lemma \ref{lemma:Arnosti-transcnew} to find $t' \in T$ such that $t_1 + qt' + \pf \in T/\pf$ is transcendental over $S/(\pf \cap S)$ for every $\pf \in E.$  If $F$ is not empty, use Lemma \ref{lemma:Arnosti-transcnew} to find $t'' \in T$ such that $t_1 + t''a_2 + \pf \in T/\pf$ is transcendental over $S/(\pf \cap S)$ for every $\pf \in F.$ If both $E$ and $F$ are not empty, let $u = t_1 + qt' + t''a_2$.  If $E$ is not empty and $F$ is empty, let $u = t_1 + qt'$.  If $E$ is empty and $F$ is not empty, let $u = t_1 + t''a_2$.  Finally, if both $E$ and $F$ are empty, let $u = t_1$. We claim that $S' = S[u]_{(S[u] \cap \mf)}$ is a CGIP-subring of $T$.

By Lemma \ref{localize}, we need only show that conditions \ref{twonew} and \ref{threenew} of Definition \ref{defn:GIPNew} hold for $S[u]$. If $E$ and $F$ are both empty, then all minimal prime ideals of $T$ are singular, and so both conditions hold. So suppose $\Min(T)$ contains elements that are nonsingular. The forward direction of condition \ref{twonew} follows from the same argument that was used in the base case of the induction. For the reverse direction, observe that for all $\pf\in E$ we have $u + \pf = t_1 + qt' + t''a_2 + \pf = t_1 + qt' + \pf$ which is transcendental over $S/(\pf \cap S)$ by construction. Similarly, for $\pf\in F$, we have $u + \pf = t_1 + t'q + t''a_2 + \pf = t_1 + t''a_2 + \pf$ is transcendental over $S/(\pf \cap S)$. Now suppose $\pf,\pf' \in \Min(T)$ and there is an $i \in \{1,2, \ldots ,n\}$ such that $\pf \in \Cc_i$ and $\pf' \in \Cc_i$. Then $\pf \cap S = \pf' \cap S$. If $\pf$ or $\pf'$ is singular then $\pf = \pf'$ and we have $\pf \cap S[u] = \pf' \cap S[u]$. Assume then that $\pf$ and $\pf'$ are both nonsingular.
%Then either $\pf$ and $\pf'$ are both in $E$ or they are both in $F$. 
Then, by \Cref{lemma:Arnosti-adjoining-lemma}, $\pf \cap S[u] = \pf' \cap S[u]$. Hence, condition \ref{twonew} holds. Note that condition \ref{threenew} of Definition \ref{defn:GIPNew} holds for $S[u]$ by \Cref{lemma:Arnosti-adjoining-lemma}. It follows that $S' = S[u]_{(S[u] \cap \mf)}$ is a CGIP-subring of $T$.
%Additionally, the backwards direction of condition \ref{twonew} holds for primes $\pf, \pf'$ provided $\pf,\pf'\in E\cup F$. If $\pf\notin E\cup F$ then $\pf$ is singular, hence $\pf\in \Cc_i = \{\pf\}$ and $\pf = \pf'$. It follows that $S' = S[u]_{S[u]\cap \mf}$ is a CGIP-subring of $T$.

Let $J = (a_2, \ldots, a_m)S'$, and let $c^* = c - a_1u$.  Then $c^* \in JT \cap S'$.  By induction there is a CGIP-subring $R$ of $T$ such that $S' \subseteq R$ and $c^* \in (a_2, \ldots,a_m)R$.  So $c^* = a_2r_2 + \cdots + a_mr_m$ for some $r_i \in R$. Now $c = a_1u + a_2r_2 + \cdots + a_mr_m \in (a_1, a_2, \ldots,a_m)R = IR$.  It follows that $R$ is the desired CGIP-subring of $T$.
\end{proof}

In the next lemma we show that if $R$ is a CGIP-subring of $T$ and $\qf \in \Spec(T)$ is not a minimal prime ideal of $T$ that is nonsingular then there exists another CGIP-subring $S$ of $T$ such that $R \subseteq S$ and $S$ contains a generating set for $\qf$. This lemma will be crucial in showing that our final ring is excellent. We note that the lemma is a modification of Lemma 3.6 in \cite{countable-domain}.

\begin{lemma}\label{lemma:find-gen-setnew}
Let $(T,\mf)$ be a complete local ring containing the rationals and suppose $\dim T \geq 1$. Let $\Pc = \{\Cc_1, \dots, \Cc_n\}$ be a partition on $\Min(T)$ such that if $\pf\in \Cc_i$ for some $i \in \{1,2, \ldots ,n\}$ and $\pf$ is singular then $\Cc_i = \{\pf\}$.
Let $\qf\in \Spec(T)$ and suppose $\qf$ is either singular or not minimal. If $(R, R \cap \mf)$ is a  CGIP-subring of $T$, there exists a CGIP-subring $(S, S \cap \mf)$ of $T$ such that $R\subseteq S$ and $S$ contains a generating set for $\qf$. 
\end{lemma}
\begin{proof}
Let $ \qf = (x_1, \dots, x_m)$, and so $\{x_1, x_2, \ldots, x_m\}$ is a generating set for $\qf$. We define a chain of CGIP-subrings $R = R_1\subseteq \dots \subseteq R_{m+1}$ such that $R_{m+1}$ contains a generating set for $\qf$. To construct $R_{i+1}$ from $R_i$, we show there exists an element $\tilde{x_i}$ of $T$ such that we can replace $x_i$ with $\tilde{x_i}$ in the generating set and that $R[\tilde{x_i}]_{(R[\tilde{x_i}]\cap \mf)}$ is a CGIP-subring of $T$.

Let $E$ denote the set of nonsingular minimal prime ideals of $T$, and assume $E$ is not empty. Since $\qf\notin E$, we have by Prime Avoidance that  $\bigcup_{\pf\in E}\pf\not\supset \qf$, and thus there exists $y\in \qf$ that is contained in no element of $E$. We now apply \Cref{lemma:Arnosti-transcnew} with $C = E, J = \mf, t = x_1, q = y$, to find an element $\alpha \in \mf$ such $x_1 + \alpha y + \pf \in T/\pf$ is transcendental over $R_1/(R_1\cap \pf)$ for each nonsingular minimal prime ideal $\pf$ of $T$. Define $\tilde{x_1} = x_1 + \alpha y$ if $E$ is not empty and $\tilde{x_1} = x_1$ if $E$ is empty. We now show that $(\tilde{x_1}, x_2, \dots, x_m) = \qf$. Since $\tilde{x_1}-x_1  \in \mf \qf$, we have $(\tilde{x_1}, x_2, \dots, x_m) + \mf\qf = \qf$, and hence we conclude by Nakayama's lemma that $\qf = (\tilde{x_1}, \dots, x_m)$. Moreover, since $\tilde{x_1} + \pf$ is transcendental over $R_1/(R_1\cap \pf)$ for each nonsingular minimal prime ideal $\pf$, we have by Lemma \ref{lemma:Arnosti-adjoining-lemma} and Lemma \ref{localize} that $R_2 = R_1[\tilde{x_1}]_{(R_1[\tilde{x_1}]\cap \mf)}$ is a CGIP-subring of $T$. Repeat this process with $R_1$ replaced by $R_2$ and $x_1$ replaced by $x_2$ to obtain the CGIP-subring $R_3 = R_2[\tilde{x_2}]_{(R_2[\tilde{x_2}]\cap \mf)}$ of $T$. Continue this process until we arrive at the CGIP-subring $S = R_{m + 1}$ of $T$, and note that $S$ contains a generating set for $\qf$.
\end{proof}

Our final ring, as well as some of our intermediary rings, will be a union of a countable chain of CGIP-subrings of $T$.  In the next lemma, we show that such a union is a CGIP-subring of $T$.

\begin{lemma}\label{lemma:Arnosti-unioningnew}
Let $(T,\mf)$ be a complete local ring with $\dim T \geq 1$. Let $\Pc = \{\Cc_1, \dots, \Cc_n\}$ be a partition on $\Min(T)$ such that if $\pf\in \Cc_i$ for some $i \in \{1,2, \ldots ,n\}$ and $\pf$ is singular then $\Cc_i = \{\pf\}$.
Let $(R_i)_{i\in \Z^+}$ be a countable chain of CGIP-subrings of $T$. Then $R:=\bigcup_{i\in \Z^+}R_i$ is a CGIP-subring of $T$.
\end{lemma}

\begin{proof}
Since each $R_i$ is countably infinite, so is $R$.  Since, for every $i \in \Z^+$, the maximal ideal of $R_i$ is $R_i \cap \mf$, the maximal ideal of $R$ is $R \cap \mf$. Now let $\pf,\pf' \in \Min(T)$.  Suppose $\pf \cap R = \pf' \cap R$.  Then $\pf \cap R_1 = \pf' \cap R_1$ and so there is a $j \in \{1,2, \ldots ,n\}$ such that $\pf \in \Cc_j$ and $\pf' \in \Cc_j$.  Conversely, suppose there is a $j \in \{1,2, \ldots ,n\}$ such that $\pf \in \Cc_j$ and $\pf' \in \Cc_j$.  Then, for each $i \in \Z^+$, $\pf \cap R_i = \pf' \cap R_i$.  It follows that $\pf \cap R = \pf' \cap R$.  Finally, let $\pf \in \Min(T)$ be nonsingular and let $a \in \pf \cap R$.  Then $a \in \pf \cap R_i$ for some $i$ and so $\Ann_T(a) \not\subseteq \pf$.
\end{proof}

We use the next lemma to show that we can construct CGIP-subrings of $T$ whose completion is $T$.

\begin{lemma}\label{lemma:Arnosti-image-3new}
Let $(T,\mf)$ be a complete local ring containing the rationals and suppose $\dim T \geq 1$. Let $\Pc = \{\Cc_1, \dots, \Cc_n\}$ be a partition on $\Min(T)$ such that if $\pf\in \Cc_i$ for some $i \in \{1,2, \ldots ,n\}$ and $\pf$ is singular then $\Cc_i = \{\pf\}$.
Let $J$ be an ideal of $T$ such that $J\not\subseteq \pf$ for all $\pf \in \Min(T)$ and let $u+J\in T/J$. If $R$ is a CGIP-subring of $T$, then there exists a CGIP-subring $S$ of $T$ such that:
\begin{enumerate}
    \item $R\subseteq S\subset T$;
    \item $u+J$ is in the image of $\pi: S\to T/J$;
    \item For every finitely generated ideal $I$ of $S$, we have $IT\cap S = I$.
\end{enumerate}
\end{lemma}
\begin{proof}
Apply Lemma \ref{lemma:Arnosti-transcnew} with $q = 1$ to obtain an element $t' \in J$ such that $u + t' + \pf \in T/\pf$ is transcendental over $R/(\pf \cap R)$ for all $\pf \in \Min(T)$.  Then, using \Cref{lemma:Arnosti-adjoining-lemma} and \Cref{localize}, we have that $R' = R[u + t']_{(
R[u + t'] \cap \mf)}$ is a CGIP-subring of $T$.  Our final CGIP-subring $S$ of $T$ will contain $R'$ and so, since $u + t' \in u + J$, it will follow that $u+J$ is in the image of $\pi: S\to T/J$ and $R \subseteq S$.  Now let $$\Omega = \{(I,c) \, | \, I \mbox{ is a finitely generated ideal of } R' \mbox{ and } c \in IT \cap R'\}.$$ Since $R'$ is countable, $\Omega$ is countable.  Well order $\Omega$ using the nonnegative integers.  We inductively define a family of CGIP-subrings of $T$, one for every nonnegative integer.  Let $R_0 = R'$.  Assume that $R_j$ has been defined, and let $(I,c)$ be the $j$th element of $\Omega$. Use Lemma \ref{lemma:CGIP-close-upnew} to define the CGIP-subring $R_{j + 1}$ of $T$ so that $R_j \subseteq R_{j + 1}$ and $c \in IR_{j + 1}$.  Now let $S_1 = \bigcup_{i = 1}^{\infty} R_i$.  By Lemma \ref{lemma:Arnosti-unioningnew}, $S_1$ is a CGIP-subring of $T$. Let $I$ be a finitely generated ideal of $R'$.  By construction, if $c \in IT \cap R'$, then $c \in IS_1$.  It follows that $IT \cap R' \subseteq IS_1$ for every finitely generated ideal $I$ of $R'$. Repeat the argument with $R'$ replaced by $S_1$ to find a CGIP-subring $S_2$ of $T$ such that $S_1 \subseteq S_2$ and, if $I$ is a finitely generated ideal of $S_1$ then $IT \cap S_1 \subseteq IS_2$.  Repeat to obtain a chain $S_1 \subseteq S_2 \subseteq \cdots $ such that, if $j$ is a nonnegative integer and $I$ is a finitely generated ideal of $S_j$, then $IT \cap S_j \subseteq IS_{j + 1}$.

Let $S = \bigcup_{i = 1}^{\infty}S_i$. By Lemma \ref{lemma:Arnosti-unioningnew}, $S$ is a CGIP-subring of $T$. Now let $I = (s_1, \ldots ,s_m)$ be a finitely generated ideal of $S$, and let $c \in IT \cap S$.  Choose $j$ so that $c, s_1, \ldots ,s_m \in R_j$.  Then $c \in (s_1, \ldots ,s_m)T \cap R_j \subseteq (s_1, \ldots ,s_m)R_{j + 1} \subseteq IS$.  It follows that $IT \cap S = I$.
\end{proof}

In the next lemma, we show that if $R$ is a CGIP-subring of $T$, then there is a CGIP-subring of $T$ that contains $R$ and has $T$ as its completion.  Note that Lemma \ref{lemma:gen-covernew} is analogous to Lemma \ref{Lemma:Precompletion-Cover} from Section \ref{section-3}.

\begin{lemma}\label{lemma:gen-covernew}
Let $(T,\mf)$ be a complete local ring containing the rationals and suppose $\dim T \geq 1$ and $T/\mf$ is countable.  Let $\Pc = \{\Cc_1, \dots, \Cc_n\}$ be a partition on $\Min(T)$ such that if $\pf\in \Cc_i$ for some $i \in \{1,2, \ldots ,n\}$ and $\pf$ is singular then $\Cc_i = \{\pf\}$.
Let $(R, R \cap \mf)$ be a CGIP-subring of $T$. Then there exists a local CGIP-subring $(S, S \cap \mf)$ of $T$ such that $R\subseteq S\subseteq T$ and $\hat{S} = T$.
\end{lemma}

\begin{proof}
Use the proof of Lemma \ref{Lemma:Precompletion-Cover} with Lemma \ref{lemma:Arnosti-image-3} replaced by Lemma \ref{lemma:Arnosti-image-3new} and Lemma \ref{lemma:Arnosti-unioning} replaced by Lemma \ref{lemma:Arnosti-unioningnew}.
\end{proof}

%NOTE:  put in some commentary about how the singular locus of an excellent ring is closed and so it is equal to the variety of some ideal.  And mention that complete local rings are excellent.
%BEN: How's this?

We now focus on making our final ring excellent.  To do this, we construct our final ring so that it contains a generating set for each element of a carefully chosen countable set of prime ideals of $T$.  We start by showing our special set of prime ideals of $T$ is countable and contains no nonsingular minimal prime ideals of $T$.

Recall that all local rings are excellent. By this fact and \Cref{prop:closed-variety}, the singular locus of a complete local ring is equal to the variety of one of its ideals.  If $T$ is a complete local ring and $J$ is an ideal of $T$, then $T/J$ is a complete local ring, and so its singular locus is closed.  Therefore, the singular locus of $T/J$, which we denote as $\Sing(T/J)$, is equal to the variety of some ideal $I/J$ of $T/J$ where $I$ is an ideal of $T$ containing $J$.  We use this fact in the following lemma.

\begin{lemma}\label{Lemma:Countable-Singularitiesnew}
Let $R$ be a countable Noetherian subring of the complete local ring $T$. Define
\[
\mathfrak B_R = \bigcup_{\pf\in \Spec(R)}\{\qf \in \Spec(T): \qf\in \Min(I) \text{ where } \Sing(T/\pf T) = V(I/\pf T)\}
.\]
Then $\mathfrak B_R$ is countable and contains no nonsingular minimal prime ideal of $T$.
\end{lemma}

\begin{proof}
Since $R$ is countable and Noetherian, the set $\Spec(R)$ is also countable. Since $T$ is Noetherian, $\Min(I)$ is finite for all ideals $I$ of $T$. We conclude $\mathfrak B_R$ is countable. Suppose $\qf$ is a minimal prime ideal of $T$ that is nonsingular and is contained in $\mathfrak B_R$. Then $\qf/\pf\in \Sing(T/\pf T)$ for some $\pf \in \Spec(R)$ and  $T_\qf$ is a field. We have $\pf T_\qf\subseteq \qf T_\qf = (0)T_\qf$. But then $(T/\pf T)_{\qf/\pf}\cong T_\qf$ which is a field, contradicting the fact that $\qf/\pf\in \Sing(T/\pf T)$. Thus $\mathfrak B_R$ contains no nonsingular minimal prime ideal of $T$.
\end{proof}

We now have the tools we need to construct our final ring.  In Proposition \ref{prop:partition-precompletionnew}, we construct our ring, and in Corollary \ref{cor:sufficient-conditionnew}, we show that we can choose the partition on $\Min(T)$  so that the ring given in Proposition \ref{prop:partition-precompletionnew}  satisfies our desired properties.

\begin{proposition}\label{prop:partition-precompletionnew}
Let $(T,\mf)$ be a complete local ring containing the rationals and suppose $\dim T \geq 1$ and $T/\mf$ is countable.  Let $\Pc = \{\Cc_1, \dots, \Cc_n\}$ be a partition on $\Min(T)$ such that if $\pf\in \Cc_i$ for some $i \in \{1,2, \ldots ,n\}$ and $\pf$ is singular then $\Cc_i = \{\pf\}$.  Suppose also that if $\pf, \pf'\in \Cc_i$ for some $i$, then ${\dim T/\pf} = {\dim T/\pf'}$.
Then there exists a countable excellent local subring $A$ of $T$ such that $\hat{A} = T$ and, for $\pf, \pf' \in \Min(T)$, we have $\pf \cap A = \pf' \cap A$ if and only if there is an $i \in \{1,2, \ldots ,n\}$ such that $\pf \in \Cc_i$ and $\pf' \in \Cc_i$.
\end{proposition}

\begin{proof}
To start, let $R_0$ be the CIP-subring of $T$ guaranteed by \Cref{prop:base-ring}, and note that $R_0$ is a CGIP-subring of $T$. Let $S_0$ be the CGIP-subring obtained from Lemma \ref{lemma:gen-covernew} so that  $R_0\subseteq S_0$ and $\widehat{S_0} = T$. We now describe how the ring $S_{i+1}$ will be constructed from $S_i$ for any $i\geq 0$, which will closely follow an argument given in \cite{countable-domain}. Enumerate $\mathfrak B_{S_i} = (\qf_i)_{i\in \Z^+}$ where $\mathfrak B_{S_i}$ is defined in the statement of Lemma \ref{Lemma:Countable-Singularitiesnew}. Let $R_i^0 = S_i$ and for each $j\in \Z^+$, let $R_i^j$ be the ring constructed from $R_i^{j-1}$ by adjoining a generating set for $\qf_j$ using \Cref{lemma:find-gen-setnew}. Then we let $R_{i}^\infty = \bigcup_{j=0}^\infty R_i^j$. By \Cref{lemma:Arnosti-unioningnew,Lemma:Countable-Singularitiesnew}, this is a CGIP-subring of $T$. Finally, let $S_{i + 1}$ be the CGIP-subring of $T$ obtained from Lemma \ref{lemma:gen-covernew} so that $R_i^\infty\subseteq S_{i+1}$ and $\widehat{S_{i + 1}} = T$. We claim that $A = \bigcup_{i=0}^\infty S_i$ satisfies each of the desired properties.

First, note that $A$ is a CGIP-subring by \Cref{lemma:Arnosti-unioningnew}, and hence it is countable. Since, by Proposition \ref{proposition:the-machine}, $S_i\to T/\mf^2$ is surjective for each $i$, we have also that $A\to T/\mf^2$ is surjective. Proposition \ref{proposition:the-machine} also gives that, for every $i$ and for every finitely generated ideal $I$ of $S_i$, we have $IT \cap S_i = I$. By \Cref{lemma:upwards-closed}, we have $IT\cap A = I$ for every finitely-generated ideal $I$ of $A$. It follows by \Cref{proposition:the-machine} that $A$ is Noetherian and $\hat{A}=T$. Since $A$ is a CGIP-subring of $T$, we have that for $\pf, \pf' \in \Min(T)$, $\pf \cap A = \pf' \cap A$ if and only if there is an $i \in \{1,2, \ldots ,n\}$ such that $\pf \in \Cc_i$ and $\pf' \in \Cc_i$.

It remains to show that $A$ is excellent. Suppose $\pf\in \Spec(A),\qf\in \Spec(T)$ such that $\qf\cap A = \pf$. We show that $(T/\pf T)_\qf$ is a regular local ring based on an argument given in \cite{countable-domain}. Suppose, for the sake of contradiction, that $(T/\pf T)_\qf$ is not a regular local ring. Then we have $\qf/\pf T \in \Sing(T/\pf T) = V(I/\pf T)$ for some ideal $I$ of $T$. It follows that $\qf\supset \rf\supset I$ for some prime ideal $\rf$ of $T$ which is minimal over $I$. Note that $\rf/\pf T \in V(I/\pf T) = \Sing(T/\pf T)$ and $\rf\cap A = \pf$. Let $\pf = (a_1, \ldots ,a_m)$.  Choose $i$ so that $a_j \in S_i$ for all $j = 1,2, \ldots ,m$ and let $\pf_i = \pf \cap S_i$. Then $T/\pf_iT = T/\pf T$. It follows by construction that $S_{i + 1}$ contains a generating set for $\rf$ and hence so does $A$. But then $(T/\pf T)_\rf = (T/(\rf \cap A)T)_\rf = (T/\rf T)_\rf$, a field. This contradicts that $\rf/\pf T\in V(I/\pf T) = \Sing(T/\pf T)$, and hence $(T/\pf T)_\qf$ must be a regular local ring. By \Cref{lemma:excellency-over-Q}, $S$ is quasi-excellent. We now show $A$ is formally catenary;  it then follows from \cite{Ratliff}, Theorem 2.6 that $A$ is universally catenary. To do this, we check that $\widehat{(A/\pf)}$ is equidimensional for every $\pf\in \Min(A)$. Suppose $\pf\in \Min(A)$ and note that $\Min(\widehat{(A/\pf)}) = \Min(T/\pf T) = \{\qf/\pf T \, | \, \qf \in \Min(T) \mbox{ and } \qf \cap A = \pf\}$. By the properties of $A$, there is an $i \in \{1,2, \ldots ,n\}$ such that $\{\qf/\pf T \, | \, \qf \in \Min(T) \mbox{ and } \qf \cap A = \pf\} = \{\qf/\pf T \, | \, \qf \in \Cc_i\}$. By hypothesis, $\dim(T/\qf) = \dim(T/\qf')$ for any $\qf,\qf'\in \mathcal C_i$. Therefore $\widehat{A/\pf}$ is equidmensional, and by \cite{matsumura_1987}, Theorem 31.6 it follows $A/\pf$ is universally catenary. Thus $A$ is formally catenary. Since $A$ is universally catenary and quasi-excellent, $A$ is excellent. 
\end{proof}

\begin{corollary}\label{cor:sufficient-conditionnew}
Let $T$ be a complete local ring containing the rationals and let $k$ be a positive integer. If $T/\mf$ is countable and $|d_1(T)|\leq k\leq |\Min(T)|$ where $d_1(T)$ is defined in \Cref{defn:equidim-params}, then there exists a countable excellent local subring $S$ of $T$ such that $\hat{S} = T$ and $S$ has $k$ minimal prime ideals.
\end{corollary}
\begin{proof}
If $\dim T = 0$ then, since $T$ is Noetherian, it is also Artinian. Thus, there exists $N>0$ such that $\mf^N = (0)$. Since $T$ contains $\Q$, $T/\mf$ is infinite. We then have $|T| = |T/\mf^N| = |T/\mf|$ by \Cref{lemma:cardinality-m-2}. Furthermore, $T$ is complete, hence excellent, and so $S = T$ will suffice. 
% NOTE:  I think there are problems with the first sentence -- consider $\Z_2[[x]]/(x^2)$.  Fix this!! BEN: This implicitly uses the fact that $\Q\subseteq T$. If we add the phrase ``and the fact that $T$ contains $\Q$, hence is infinite'' right after Lemma 2.6, I believe this fixes this issue.

Now suppose $\dim T\geq 1$. Choose a partition $\Pc = \{\Cc_1, \dots, \Cc_k\}$ on $\Min(T)$ satisfying the following two conditions.
\begin{enumerate}
    \item if $\pf\in \Cc_i$ for some $i \in \{1,2, \ldots ,k\}$ and $\pf$ is singular then $\Cc_i = \{\pf\}$
    \item if $\pf, \pf'\in \Cc_i$ for some $i$, then ${\dim T/\pf} = {\dim T/\pf'}$
\end{enumerate}
 Note that, given the conditions on $k$, it is possible to choose such a partition. Now use Proposition \ref{prop:partition-precompletionnew} to find a countable excellent local subring $S$ of $T$ such that $\hat{S} = T$ and, for $\pf, \pf' \in \Min(T)$, we have $\pf \cap S = \pf' \cap S$ if and only if there is an $i \in \{1,2, \ldots ,k\}$ such that $\pf \in \Cc_i$ and $\pf' \in \Cc_i$. Since $T$ is a faithfully flat extension of $S$, if $\qf \in \Min(S)$, there is a $\pf \in \Min(T)$ such that $\pf \cap S = \qf$. It follows that $S$ has $k$ minimal prime ideals.
\end{proof}
\label{temp-label}
In fact, our result is slightly stronger than stated. The value $k$ can be viewed as an abstraction; what these results actually show is which maps $\Min(T)\to \Min(S)$ are possible. 

The next proposition proves the second part of our main theorem.

\begin{proposition}\label{prop:necessary-conditionnew}
Suppose $A$ is a countable excellent local ring with $k$ minimal prime ideals and let $\hat{A} = T$. If $\mf$ is the maximal ideal of $T$ then $T/\mf$ is countable and $|d_2(T)|\leq k\leq |\Min(T)|$ where $d_2(T)$ is defined in Definition \ref{defn:equidim-params}.
\end{proposition}

\begin{proof}
Suppose $A$ is a countable excellent local ring with $\widehat A = T$. We show that $|d_2(T)|\leq \abs{\Min(A)}\leq \abs{\Min(T)}$. The second inequality holds since $T$ is a faithfully flat extension of $A$. Now let $\qf, \qf' \in \Min(\pf T)$ for some $\pf\in \Min(A)$. Then since $A$ is universally catenary, it is formally catenary, and thus $\widehat{(A/\pf)} = T/\pf T$ is equidimensional. It follows that $\dim(T/\qf T) = \dim(T/\qf' T)$ and the first inequality follows. Finally, since $A$ is countable and ${A/(\mf\cap A)} \cong T/\mf$, we conclude that $T/\mf$ is also countable.
\end{proof} 

We now restate our main theorem, which follows from the last two results.

\begin{theorem*}[\ref{Thm:ctbl-excellent-gen}]
Let $(T, \mf)$ be a complete local ring containing the rationals and let $k$ be a positive integer. 
\begin{itemize}
    \item If both of the following conditions hold, then there exists a countable excellent local ring $A$ such that $\hat{A} = T$ and $|\Min(A)| = k$.
    \begin{enumerate}
        \item $T/\mf$ is countable;
        \item $|d_1(T)|\leq k\leq |\Min(T)|.$
    \end{enumerate}
    \item If there exists a countable excellent local ring $A$ such that $\hat{A} = T$ and $|\Min(A)| = k$, then
    \begin{enumerate}
        \item $T/\mf$ is countable and
        \item $|d_2(T)|\leq k\leq |\Min(T)|$.
    \end{enumerate}
\end{itemize}
\end{theorem*}

\begin{proof}
The first statement follows from Corollary \ref{cor:sufficient-conditionnew} and the second statement follows from Proposition \ref{prop:necessary-conditionnew}.
\end{proof}

For some rings we have $|d_2(T)| < |d_1(T)|$, and for these rings \Cref{Thm:ctbl-excellent-gen} does not completely characterize the possible minimal spectra for the ring $A$. Nevertheless, this result \textit{does} allow us to prove two characterization-style results.

\begin{cor*}[\ref{cor:ctbl-excellent-red}]
Let $(T,\mf)$ be a complete local ring containing the rationals and satisfying Serre's $R_0$ condition. Let $k$ be a positive integer. Then $T$ is the completion of a countable excellent local ring with $k$ minimal prime ideals if and only if $T/\mf$ is countable and $|d_2(T)| \leq k\leq |\Min(T)|$.
\end{cor*}
\begin{proof}
Since $T_\qf$ is a field for every $\qf\in \Min(T)$, we have $|d_1(T)| = |d_2(T)|$. 
\end{proof}

\begin{cor*}[\ref{cor:ctbl-excellent-char}]
Let $(T,\mf)$ be a complete local ring containing the rationals. The following are equivalent.
\begin{enumerate}
    \item $T$ is the completion of a countable excellent local ring with $|\Min(T)|$ minimal prime ideals;
    \item $T$ is the completion of a countable excellent local ring;
    \item $T/\mf$ is countable.
\end{enumerate}
\end{cor*}
\begin{proof}
We have (1) implies (2) trivially. We then have (2) implies (3) by \Cref{prop:necessary-conditionnew}. Finally, (3) implies (1) by \Cref{cor:sufficient-conditionnew} and the fact that $|d_1(T)|\leq |\Min(T)|$.
\end{proof}
\printbibliography
\end{document}